\documentclass[11pt]{amsart}
\usepackage{amssymb}
\usepackage{epsfig}
\usepackage{mathdots}
 \theoremstyle{plain}
\newtheorem{theorem}{Theorem}
\newtheorem{corollary}{Corollary}

\newtheorem{lemma}{Lemma}
\newtheorem{proposition}{Proposition}
\newtheorem{example}{Example}
\theoremstyle{definition}
\newtheorem{definition}{Definition}
\theoremstyle{remark}

\numberwithin{equation}{section}

\newcommand{\bT}{\begin{theorem}}
\newcommand{\eT}{\end{theorem}}
\newcommand{\bProp}{\begin{proposition}}
\newcommand{\eProp}{\end{proposition}}
\newcommand{\bE}{\begin{example}}
\newcommand{\eE}{\end{example}}
\newcommand{\bL}{\begin{lemma}}
\newcommand{\eL}{\end{lemma}}
\newcommand{\bP}{\begin{proof}}
\newcommand{\eP}{\end{proof}}
\newcommand{\bC}{\begin{corollary}}
\newcommand{\eC}{\end{corollary}}
\newcommand{\bD}{\begin{definition}}
\newcommand{\eD}{\end{definition}}
\newcommand{\be}{\begin{enumerate}}
\newcommand{\ee}{\end{enumerate}}
\newcommand{\beqa}{\begin{eqnarray*}}
\newcommand{\eeqa}{\end{eqnarray*}}
\newcommand{\beqaa}{\begin{eqnarray}}
\newcommand{\eeqaa}{\end{eqnarray}}
\newcommand{\ba}{\begin{array}}
\newcommand{\ea}{\end{array}}

\setbox0=\hbox{$+$}
\newdimen\plusheight
\plusheight=\ht0
\def\+{\;\lower\plusheight\hbox{$+$}\;}

\setbox0=\hbox{$-$}
\newdimen\minusheight
\minusheight=\ht0
\def\-{\;\lower\minusheight\hbox{$-$}\;}

\setbox0=\hbox{$\cdots$}
\newdimen\cdotsheight
\cdotsheight=\plusheight
\def\cds{\lower\cdotsheight\hbox{$\cdots$}}
\begin{document}

\title[Some new Transformations for Bailey Pairs ]
       {Some new Transformations for Bailey pairs and WP-Bailey Pairs}

\author{James Mc Laughlin}
\address{Mathematics Department\\
25 University Avenue\\
West Chester University, West Chester, PA 19383}
\email{jmclaughl@wcupa.edu}

%\author{Peter Zimmer}
%\address{Mathematics Department\\
% Anderson Hall\\
%West Chester University, West Chester, PA 19383}
%\email{pzimmer@wcupa.edu}


\date{\today}

\begin{abstract}
We derive several new transformations relating WP-Bailey pairs. We
also consider the corresponding transformations relating standard Bailey
pairs, and as a consequence, derive some quite general expansions
for products of theta functions which can also be expressed as
certain types of Lambert series.
\end{abstract}

\maketitle

\section{Introduction}

Andrews \cite{A01}, building on previous work of Bressoud
\cite{B81a}
 and Singh \cite{S94}, defined a \emph{WP-Bailey
pair} to be a pair of sequences $(\alpha_{n}(a,k,q)$,
$\beta_{n}(a,k,q))$ (if the context is clear, we occasionally
suppress the dependence on some or all of $a$, $k$ and $q$) satisfying $\alpha_{0}(a,k,q)$
$=\beta_{0}(a,k,q)=1$, and for $n>0$, {\allowdisplaybreaks
\begin{align}\label{WPpair}
\beta_{n}(a,k,q) &= \sum_{j=0}^{n}
\frac{(k/a;q)_{n-j}(k;q)_{n+j}}{(q;q)_{n-j}(aq;q)_{n+j}}\alpha_{j}(a,k,q).
\end{align}
}If $k=0$, then the pair of sequences $(\alpha_n(a,q),\beta_n(a,q))$ is called a \emph{Bailey pair with respect to $a$}.

In the same paper Andrews showed  that if the pair
$(\alpha_{n}(a,k),\,\beta_{n}(a,k))$ satisfies \eqref{WPpair}, then so
does $(\alpha_{n}'(a,k),\,\beta_{n}'(a,k))$  where
{\allowdisplaybreaks
\begin{align}\label{wpn1}
\alpha_{n}'(a,k)&=\frac{(y, z;q)_n}{(aq/y,
aq/z;q)_n}\left(\frac{k}{c}\right)^n\alpha_{n}(a,c),\\
\beta_{n}'(a,k)&=\frac{(k y/a,k z/a;q)_n}{(aq/y,
aq/z;q)_n} \notag\\
%\phantom{as}
\times \sum_{j=0}^{n} &\frac{(1-c
q^{2j})(y,z;q)_j(k/c;q)_{n-j}(k;q)_{n+j}}{(1-c)(ky/a,kz/a;q)_n(q;q)_{n-j}(qc;q)_{n+j}}
\left(\frac{k}{c}\right)^j\beta_{j}(a,c), \notag
\end{align}
}with  $c=ky z/aq$. Andrews \cite{A01} also described  a second method for deriving new WP-Bailey pairs from existing pairs, but this second method will not concern us in the present paper.

 These two constructions allow a ``tree"
of WP-Bailey pairs to be generated from a single WP-Bailey pair.
The implications of these two branches were further investigated by Andrews and Berkovich in \cite{AB02}. Spiridonov
\cite{S02} derived an elliptic generalization of Andrews first
WP-Bailey chain. Four additional branches were added to
the WP-Bailey tree by Warnaar \cite{W03}, two of which had generalizations to the
elliptic level. More recently,  Liu and Ma \cite{LM08} introduced the idea of a general
WP-Bailey chain, and
added one new branch to the WP-Bailey tree. In \cite{MZ09}, the authors added three new WP-Bailey chains.

It is not difficult to show (see Corollary 1 in \cite{MZ09b}, for example) that
the WP-Bailey chain at \eqref{wpn1} implies that if $(\alpha_n, \beta_n)$ satisfy \eqref{WPpair}, then
subject to suitable convergence conditions, {\allowdisplaybreaks
\begin{multline}\label{wpeq}
\sum_{n=0}^{\infty} \frac{(q\sqrt{k},-q\sqrt{k}, y,z;q)_{n}}
{(\sqrt{k},-\sqrt{k}, q k/y,q k/z;q)_{n}}\left( \frac{q a}{y z }\right )^{n} \beta_n =\\
\frac{(q k,q k/yz,q a/y,q a/z;q)_{\infty}} {(q k/y,q k/z,q a,q
a/yz;q)_{\infty}} \sum_{n=0}^{\infty}\frac{(y,z;q)_{n}}{(q a/y ,q
a/z;q)_n}\left (\frac{q a}{y z}\right)^{n}\alpha_n.
\end{multline}
}

%Each of the other WP-Bailey chains implies a similar relation for
%WP-Bailey pairs, but these will not concern us in the present paper.

In the present paper we prove some new relations for WP-Bailey
pairs. These include the following.

\begin{theorem}\label{t1}
If $(\alpha_n(a,k), \beta_n(a,k))$ is a WP-Bailey pair, then subject
to suitable convergence conditions, {\allowdisplaybreaks
\begin{multline}\label{wpeq8}
\sum_{n=1}^{\infty} \frac{(q\sqrt{k},-q\sqrt{k},z;q)_{n}(q;q)_{n-1}}
{\left(\sqrt{k},-\sqrt{k}, q k,\frac{q k}{z};q\right)_{n}}\left( \frac{q a}{ z }\right )^{n} \beta_n(a,k)\\
- \sum_{n=1}^{\infty}
\frac{\left(q\sqrt{\frac{1}{k}},-q\sqrt{\frac{1}{k}},\frac{1}{z};q\right)_{n}(q;q)_{n-1}}
{\left(\sqrt{\frac{1}{k}},-\sqrt{\frac{1}{k}}, \frac{q }{k},\frac{q
z}{k};q\right)_{n}}\left( \frac{q z}{ a }\right )^{n}
\beta_n\left(\frac{1}{a},\frac{1}{k}\right) -\\
 \sum_{n=1}^{\infty}\frac{(z;q)_{n}(q;q)_{n-1}}{\left(q a ,\frac{q
a}{z};q\right)_n}\left (\frac{q a}{z}\right)^{n}\alpha_n(a,k)
%\\
+
\sum_{n=1}^{\infty}\frac{\left(\frac{1}{z};q\right)_{n}(q;q)_{n-1}}{\left(\frac{q}{a}
,\frac{q z}{a};q\right)_n}\left (\frac{q
z}{a}\right)^{n}\alpha_n\left(\frac{1}{a},\frac{1}{k}\right)\\=
\frac{(a-k)\left(1-\frac{1}{z}\right)\left(1-\frac{ak}{z}\right)}
{(1-a)(1-k)\left(1-\frac{a}{z}\right)\left(1-\frac{k}{z}\right)}
+\frac{z}{k}\frac{\left(z,\frac{q}{z},\frac{k}{a},
\frac{qa}{k},\frac{ak}{z},\frac{qz}{ak},q,q;q\right)_{\infty}}
{\left(\frac{z}{k},\frac{qk}{z},\frac{z}{a},
\frac{qa}{z},a,\frac{q}{a},k,\frac{q}{k};q\right)_{\infty}}.
\end{multline}
}
\end{theorem}

\begin{theorem}\label{c3}
If $(\alpha_n(a,k,q), \beta_n(a,k,q))$ is a WP-Bailey pair, then subject
to suitable convergence conditions, {\allowdisplaybreaks
\begin{multline}\label{wpeq2n}
\sum_{n=1}^{\infty} \frac{(1-k q^{2n})(z;q)_{n}(q;q)_{n-1}}
{(1-k)( q k,q k/z;q)_{n}}\left( \frac{q a}{ z }\right )^{n} \beta_n(a,k,q)\\
+  \sum_{n=1}^{\infty} \frac{(1+k q^{2n})(z;q)_{n}(q;q)_{n-1}}
{(1+k)( -q k,-q k/z;q)_{n}}\left( \frac{-q a}{ z }\right )^{n}
\beta_n(-a,-k,q)\\-2 \sum_{n=1}^{\infty} \frac{(1-k^2
q^{4n})(z^2;q^2)_{n}(q^2;q^2)_{n-1}} {(1-k^2)( q^2 k^2,q^2
k^2/z^2;q^2)_{n}}\left( \frac{q^2 a^2}{ z^2 }\right )^{n}
\beta_n(a^2,k^2,q^2)\\=
\sum_{n=1}^{\infty}\frac{(z;q)_{n}(q;q)_{n-1}}{(q a ,q
a/z;q)_n}\left (\frac{q a}{z}\right)^{n}\alpha_n(a,k,q)\\ +
\sum_{n=1}^{\infty}\frac{(z;q)_{n}(q;q)_{n-1}}{(-q a ,-q
a/z;q)_n}\left (\frac{-q a}{z}\right)^{n}\alpha_n(-a,-k,q)\\-2
\sum_{n=1}^{\infty}\frac{(z^2;q^2)_{n}(q^2;q^2)_{n-1}}{(q^2 a^2 ,q^2
a^2/z^2;q^2)_n}\left (\frac{q^2
a^2}{z^2}\right)^{n}\alpha_n(a^2,k^2,q^2).
\end{multline}
}
\end{theorem}

We find  some similar relations for standard Bailey pairs and derive
some interesting consequences. For example, recall that  $\psi(q)=\sum_{n=0}^{\infty}q^{n(n+1)/2}$ \\ $=(q^2;q^2)_{\infty}/(q;q^2)_{\infty}$ is Ramanujan's theta function (see \cite[page 36]{B91}, for example).  If $d \not=q^{3n\pm1}$,
then {\allowdisplaybreaks
\begin{multline*}
q\frac{\psi^3(q^3)}{\psi(q)}=
\sum_{n=1}^{\infty}\frac{(q^6;q^6)_{n-1}(q^2/d;q^6)_n(-q^2)^n }
{(q^2;q^6)_n(q^2/d,q^3;q^3)_n}\\
-\sum_{n=1}^{\infty}\frac{(q^6;q^6)_{n-1}(q/d;q^6)_n(-q)^n }
{(q;q^6)_n(q/d,q^3;q^3)_n}\\
+\sum_{n=1}^{\infty}\frac{(1-q^{12n-2})(q^6;q^6)_{2n-1}(1/q^2,d;q^6)_n q^{6n^2} }{(1-1/q^2)(q^2;q^6)_{2n}(q^4/d,q^6;q^6)_n d^n}\\
-\sum_{n=1}^{\infty}\frac{(1-q^{12n-1})(q^6;q^6)_{2n-1}(1/q,d;q^6)_n
q^{6n^2+3n} }{(1-1/q)(q^4;q^6)_{2n}(q^5/d,q^6;q^6)_n d^n}.
\end{multline*}}
We show that similar
results hold for many other theta products.

We use the standard notations: {\allowdisplaybreaks
\begin{align*}
           (a;q)_n &:= (1-a)(1-aq)\cdots (1-aq^{n-1}), \\
          (a_1, a_2, \dots, a_j; q)_n &:= (a_1;q)_n (a_2;q)_n \cdots (a_j;q)_n ,\\
           (a;q)_\infty &:= (1-a)(1-aq)(1-aq^2)\cdots, \mbox{ and }\\
          (a_1, a_2, \dots, a_j; q)_\infty &:= (a_1;q)_\infty (a_2;q)_\infty \cdots (a_j;q)_\infty.
\end{align*}}

We will make use of Bailey's $\,_{6}\psi_6$ summation formula
\cite{W36}. {\allowdisplaybreaks
\begin{multline}\label{baileyeq1}
\frac{ (aq,aq/bc,aq/bd,aq/be,aq/cd,aq/ce,aq/de,q,q/a;q)_{\infty} } {
(aq/b,aq/c,aq/d,aq/e,q/b,q/c,q/d,q/e,qa^2/bcde;q)_{\infty} }\\ =
\sum_{n=-\infty}^{\infty} \frac{(1-a q^{2n}) (b,c,d,e;q)_{n}}
{(1-a)(aq/b,aq/c,aq/d,aq/e;q)_{n}} \left( \frac{q
a^2}{b c d e}\right)^n\\
=1+\sum_{n=1}^{\infty} \frac{(1-a q^{2n}) (b,c,d,e;q)_{n}}
{(1-a)(aq/b,aq/c,aq/d,aq/e;q)_{n}} \left( \frac{q a^2}{b c d
e}\right)^n\\
+\sum_{n=1}^{\infty}\frac{(1-1/a q^{2n}) (b/a,c/a,d/a,e/a;q)_{n}}
{(1-1/a)(q/b,q/c,q/d,q/e;q)_{n}} \left( \frac{q a^2}{b c d
e}\right)^n,
\end{multline}
} where the second equality follows from the definition
\begin{equation}\label{qneg}
(z;q)_{-n}= \frac{ (-1)^n q ^{n(n+1)/2} }
                 { z^n(q/z;q)_n}.
\end{equation}

We also recall Jackson's summation formula for a very-well-poised $_6\phi_5$
series~\cite[p. 356, Eq. (II. 20)]{GR04} (which follows upon setting $e=a$ in \eqref{baileyeq1}):
\begin{equation}\label{6phi5}
\sum_{n=0}^{\infty}\frac{(a,q\sqrt{a},-q\sqrt{a},b,c,d;q)_n}
{\left(q,\sqrt{a},-\sqrt{a},\frac{aq}{b},\frac{aq}{c},\frac{aq}{d};q\right)_n}
\left(\frac{aq}{bcd}\right)^n=\frac{(aq,aq/bc,aq/bd,aq/cd;q)_\infty}{(aq/b,aq/c,aq/d,aq/bcd;q)_\infty}.
\end{equation}

Finally, we make use of the $q$-Binomial Theorem \cite[page 8]{GR04},
\begin{equation}\label{qbintheo}
\sum_{n=0}^{\infty}\frac{(a;q)_n}{(q;q)_n}z^n=\frac{(az;q)_{\infty}}{(z;q)_{\infty}}.
\end{equation}
For later use, we note the special cases
\begin{align}\label{bineqab}
\sum_{n=0}^{\infty}\frac{z^n}{(q;q)_n}&=\frac{1}{(z;q)_{\infty}},\\
\sum_{n=0}^{\infty}\frac{q^{n(n-1)/2}(-z)^n}{(q;q)_n}&=(z;q)_{\infty},\notag
\end{align}
which following respectively from setting $a=0$, and replacing $z$ with $z/a$ and then letting $a\to \infty$.
Unless stated otherwise, we assume $|q|<1$.

\section{Proofs of the Main Identities}

The next transformation follows easily from the identity at
\eqref{wpeq}.

\begin{lemma}\label{p1}
If $(\alpha_n(a,k,q), \beta_n(a,k,q))$ is a WP-Bailey pair, then subject to
suitable convergence conditions, {\allowdisplaybreaks
\begin{multline}\label{wpeq2}
\sum_{n=1}^{\infty} \frac{(q\sqrt{k},-q\sqrt{k},z;q)_{n}(q;q)_{n-1}}
{(\sqrt{k},-\sqrt{k}, q k,q k/z;q)_{n}}\left( \frac{q a}{ z }\right
)^{n} \beta_n(a,k,q)
\\
- \sum_{n=1}^{\infty}\frac{(z;q)_{n}(q;q)_{n-1}}{(q a ,q
a/z;q)_n}\left (\frac{q a}{z}\right)^{n}\alpha_n(a,k,q)\\=
\sum_{n=1}^{\infty} \frac{(q\sqrt{k},-q\sqrt{k},k,z,k/a;q)_{n}}
{(\sqrt{k},-\sqrt{k}, q k,q k/z,q a;q)_{n}(1-q^n)}\left( \frac{q a}{
z }\right )^{n}.
\end{multline}
}
\end{lemma}
\begin{proof}
Rewrite \eqref{wpeq} as
{\allowdisplaybreaks
\begin{multline}\label{wpeqa}
\sum_{n=1}^{\infty} \frac{(q\sqrt{k},-q\sqrt{k},
z;q)_{n}(yq;q)_{n-1}} {(\sqrt{k},-\sqrt{k}, q k/y,q
k/z;q)_{n}}\left( \frac{q a}{y z }\right )^{n} \beta_n\\- \frac{(q
k,q k/yz,q a/y,q a/z;q)_{\infty}} {(q k/y,q k/z,q a,q
a/yz;q)_{\infty}} \sum_{n=1}^{\infty}\frac{(z;q)_{n}(yq;q)_{n-1}}{(q
a/y ,q
a/z;q)_n}\left (\frac{q a}{y z}\right)^{n}\alpha_n =\\
\frac{1}{1-y} \left ( \frac{(q k,q k/yz,q a/y,q a/z;q)_{\infty}} {(q
k/y,q k/z,q a,q a/yz;q)_{\infty}} - 1\right).
\end{multline}
}
The left side of \eqref{wpeq2} follows upon letting $y\to 1$ on the left side of \eqref{wpeqa}.
From \eqref{6phi5} it can be seen that
\begin{equation}
\label{prod6phi5}
\frac{(q k,q k/yz,q a/y,q a/z;q)_{\infty}} {(q
k/y,q k/z,q a,q a/yz;q)_{\infty}} = \sum_{n=0}^{\infty}
\frac{(q\sqrt{k},-q\sqrt{k},k,y,z,k/a;q)_{n}} {(\sqrt{k},-\sqrt{k},
q k/y,q k/z,q a,q;q)_{n}}\left( \frac{q a}{y z }\right )^{n}.
\end{equation}
Upon making this substitution in the right side of \eqref{wpeqa}, we get the right side of  \eqref{wpeq2}, after setting $(y;q)_n/(1-y)=(yq;q)_{n-1}$, then  letting $y \to 1$ and finally setting $(q;q)_{n-1}/(q;q)_n =1/(1-q^n)$.
\end{proof}

For later use we note that the  series on the right side of \eqref{wpeq2} has
the following properties. We define
\begin{equation}\label{fakzqdef}
f(a,k,z,q):=\sum_{n=1}^{\infty}
\frac{(q\sqrt{k},-q\sqrt{k},k,z,k/a;q)_{n}} {(\sqrt{k},-\sqrt{k}, q
k,q k/z,q a;q)_{n}(1-q^n)}\left( \frac{q a}{ z }\right )^{n}.
\end{equation}
\begin{lemma}\label{l1}
If $f(a,k,z,q)$ is as defined at \eqref{fakzqdef}, $|q a|, |q k| <|z|$ and none of the denominators below vanish, then
\begin{equation}\label{wpeq3}
f(a,k,z,q) = -f(k,a,z,q).
\end{equation}
\end{lemma}
\begin{proof}
This follows easily upon writing
\begin{multline*}
\frac{1}{1-y} \left ( \frac{(q k,q k/yz,q a/y,q a/z;q)_{\infty}} {(q k/y,q k/z,q a,q
a/yz;q)_{\infty}} - 1\right)\\ = \frac{1}{1-y} \left ( 1 - \frac{(q a,q
a/yz,q k/y,q k/z;q)_{\infty}}{(q a/y,q a/z,q k,q k/yz;q)_{\infty}} \right)\frac{(q k,q k/yz,q a/y,q a/z;q)_{\infty}} {(q k/y,q k/z,q a,q
a/yz;q)_{\infty}}.
\end{multline*}
From the proof of Lemma \ref{p1}, it can be see that the result of letting $y\to 1$ on the left side above is $f(a,k,z,q)$. On the other hand, the infinite product following the ``$-$" sign on the right side above is the product on the left side above with $a$ and $k$ interchanged, so that the result of  letting $y \to 1$ on the right side is $-f(k,a,z,q)$.
\end{proof}

We remark in passing that the expansion at \eqref{prod6phi5} and the similar expansion of the reciprocal of this product imply that if
\[
g(a,k,y,z,q):=\sum_{n=0}^{\infty} \frac{(q\sqrt{k},-q\sqrt{k},k,y,z,k/a;q)_{n}}
{(\sqrt{k},-\sqrt{k}, q k/y,q k/z,q a,q;q)_{n}}\left( \frac{q a}{y z }\right )^{n},
\]
then
\begin{equation}\label{wpeq4}
g(a,k,y,z,q)=\frac{1}{g(k,a,y,z,q)}.
\end{equation}

We next express $f(a,k,z,q)$ as a sum of Lambert series.

\begin{lemma}\label{l2}
If $f(a,k,z,q)$ is as defined at \eqref{fakzqdef}, $|q a| <|z|$ and  none of the denominators  below vanish, then
\begin{equation}\label{wpeq5}
f(a,k,z,q)
=
\sum_{n=1}^{\infty}\frac{k q^n}{1-kq^n}+ \sum_{n=1}^{\infty}\frac{ q^n a/z}{1-q^n a/z}-
\sum_{n=1}^{\infty}\frac{a q^n}{1-aq^n}- \sum_{n=1}^{\infty}\frac{ q^n k/z}{1-q^n k/z}.
\end{equation}
\end{lemma}
\begin{proof}
If we define
\[
G(y):= \frac{(q k,q k/yz,q a/y,q a/z;q)_{\infty}} {(q k/y,q k/z,q a,q
a/yz;q)_{\infty}}
\]
we see that $G(1)=1$ and
\begin{align*}
f(a,k,z,q) &= \lim_{y \to 1}\frac{1}{1-y} \left ( \frac{(q k,q k/yz,q a/y,q a/z;q)_{\infty}} {(q k/y,q k/z,q a,q
a/yz;q)_{\infty}} - 1\right)\\
& =  \lim_{y \to 1}\frac{G(y)-G(1)}{1-y}=-G'(1).
\end{align*}
That $-G'(1)$ equals the right side of \eqref{wpeq5} follows by logarithmically differentiating the infinite products in $G(y)$,
noting that
\[
G'(1)=G(1)\frac{d \log G(y)}{dy}\bigg |_{y=1}=\frac{d \log G(y)}{dy}\bigg |_{y=1}.
\]
\end{proof}

\begin{lemma}\label{l3}
If $f(a,k,z,q)$ is as defined at \eqref{fakzqdef}, $|q a| <|z|$ and  none of the denominators  below vanish, then
\begin{equation}\label{wpeq6}
f(a,k,z,q)+f(-a,-k,z,q)
=
2f(a^2,k^2,z^2,q^2).
\end{equation}
\end{lemma}
\begin{proof}
Use  \eqref{wpeq5} to write $f(a,k,z,q)+f(-a,-k,z,q)$ in terms of Lambert series. Then use the elementary identity
\begin{multline*}
\sum_{n=1}^{\infty} \frac{x q^n}{1-x q^n}+\sum_{n=1}^{\infty} \frac{(-x)q^n}{1-(-x)q^n}=\sum_{n=1}^{\infty}\left( \frac{x q^n}{1-x q^n}+\frac{(-x)q^n}{1-(-x)q^n}\right)\\
= \sum_{n=1}^{\infty} \frac{2x^2q^{2n}}{1-x^2q^{2n}}
= 2\sum_{n=1}^{\infty} \frac{x^2q^{2n}}{1-x^2q^{2n}}
\end{multline*}
to combine pairs of Lambert series into  single Lambert series, thus deriving the right side of \eqref{wpeq6}
\end{proof}

Remark: By somewhat similar reasoning, one can show that if $m\geq 2$ is a positive integer and $\omega$ is a primitive $m$-root of unity, then
\[
\sum_{j=0}^{m-1}f(a \omega^j, k \omega ^j, z, q) = m f(a^m,k^m,z^m,q^m).
\]

\begin{proof}[Proof of Theorem \ref{c3}]
Use \eqref{wpeq2} (noting that the series on the right side is $f(a,k,z,q)$) to substitute for $f(a,k,z,q)$, $f(-a,-k,z,q)$ and $f(a^2,k^2,z^2,q^2)$ in \eqref{wpeq6}, and the results follows after a little rearrangement of terms.
\end{proof}

One could easily insert specific  WP-Bailey pairs in \eqref{wpeq2n}
to provide explicit identities, but we leave that to the reader. We also note that letting $k \to 0$ in Theorem \ref{c3} gives a result for standard Bailey pairs.

\begin{corollary}\label{c33}
If $(\alpha_n(a,q), \beta_n(a,q))$ is a Bailey pair with respect to
$a$, then subject to suitable convergence conditions,
{\allowdisplaybreaks
\begin{multline}\label{wpeq2nn}
\sum_{n=1}^{\infty} (z;q)_{n}(q;q)_{n-1} \left( \frac{q a}{ z
}\right )^{n} \beta_n(a,q) +  \sum_{n=1}^{\infty}
(z;q)_{n}(q;q)_{n-1}\left( \frac{-q a}{ z }\right )^{n}
\beta_n(-a,q)\\-2 \sum_{n=1}^{\infty}
(z^2;q^2)_{n}(q^2;q^2)_{n-1}\left( \frac{q^2 a^2}{ z^2 }\right )^{n}
\beta_n(a^2,q^2)\\=
\sum_{n=1}^{\infty}\frac{(z;q)_{n}(q;q)_{n-1}}{(q a ,q
a/z;q)_n}\left (\frac{q a}{z}\right)^{n}\alpha_n(a,q)\\ +
\sum_{n=1}^{\infty}\frac{(z;q)_{n}(q;q)_{n-1}}{(-q a ,-q
a/z;q)_n}\left (\frac{-q a}{z}\right)^{n}\alpha_n(-a,q)\\-2
\sum_{n=1}^{\infty}\frac{(z^2;q^2)_{n}(q^2;q^2)_{n-1}}{(q^2 a^2 ,q^2
a^2/z^2;q^2)_n}\left (\frac{q^2
a^2}{z^2}\right)^{n}\alpha_n(a^2,q^2).
\end{multline}
}
\end{corollary}
Once again we leave it to the reader to produce particular identities, by inserting specific Bailey pairs.

\begin{lemma}
If $f(a,k,z,q)$ is as defined at \eqref{fakzqdef},  $|q a| <|z|<|a/q|$ and  none of the denominators  vanish, then
\begin{multline}\label{wpeq7}
f(a,k,z,q)-f\left(\frac{1}{a},\frac{1}{k},\frac{1}{z},q\right)=
\frac{(a-k)(1-1/z)(1-ak/z)}{(1-a)(1-k)(1-a/z)(1-k/z)}\\
+\frac{z}{k}\frac{(z,q/z,k/a,qa/k,ak/z,qz/ak,q,q;q)_{\infty}}
{(z/k,qk/z,z/a,qa/z,a,q/a,k,q/k;q)_{\infty}}.
\end{multline}
\end{lemma}

\begin{proof}
One can check (preferably with a computer algebra system) that
\begin{multline*}
\frac{k q^n}{1-kq^n}+ \frac{ q^n a/z}{1-q^n a/z}-
\frac{a q^n}{1-aq^n}- \frac{ q^n k/z}{1-q^n k/z} \\= \frac{(a-k) q^n \left(1-\frac{a k q^{2 n}}{z}\right) (1-z)}{\left(1-a
   q^n\right) \left(1-k q^n\right) \left(1-\frac{a q^n}{z}\right)
   \left(1-\frac{k q^n}{z}\right) z}\\
   = \frac{(k-a)(1-1/z)(1-a k/z)}{(1-a)(1-k)(1-a/z)(1-k/z)}\frac{\left(1-\frac{a k q^{2 n}}{z}\right)(a,k,a/z,k/z;q)_n q^n}{(1-a k/z)(aq,kq,aq/z,kq/z;q)_n},
\end{multline*}
so that
\begin{multline*}
f(a,k,z,q)\\=\frac{(k-a)(1-1/z)(1-a k/z)}{(1-a)(1-k)(1-a/z)(1-k/z)}\sum_{n=1}^{\infty}\frac{\left(1-\frac{a k q^{2 n}}{z}\right)(a,k,a/z,k/z;q)_n q^n}{(1-a k/z)(aq,kq,aq/z,kq/z;q)_n}
\end{multline*}
Similarly, it can be shown that
\begin{multline*}
f\left(\frac{1}{a},\frac{1}{k},\frac{1}{z},q\right)=-\frac{(k-a)(1-1/z)(1-a k/z)}{(1-a)(1-k)(1-a/z)(1-k/z)}\\ \times\sum_{n=1}^{\infty}\frac{\left(1-\frac{z q^{2 n}}{a k}\right)(1/a,1/k,z/a,z/k;q)_n q^n}{(1-z/a k)(q/a,q/k,qz/a,qz/k;q)_n}.
\end{multline*}
From the remarks above and \eqref{qneg}, \begin{multline*}
f(a,k,z,q)-f\left(\frac{1}{a},\frac{1}{k},\frac{1}{z},q\right)+\frac{(k-a)(1-1/z)(1-a k/z)}{(1-a)(1-k)(1-a/z)(1-k/z)}\\
=\frac{(k-a)(1-1/z)(1-a k/z)}{(1-a)(1-k)(1-a/z)(1-k/z)}\sum_{n=-\infty}^{\infty}\frac{\left(1-\frac{a k q^{2 n}}{z}\right)(a,k,a/z,k/z;q)_n q^n}{(1-a k/z)\left(aq,kq,\frac{aq}{z},\frac{kq}{z};q\right)_n}\\
=\frac{(k-a)(1-1/z)(1-a k/z)}{(1-a)(1-k)(1-a/z)(1-k/z)}\\
\times
\frac{(akq/z,q/z,kq/a,q,q,qa/k,zq,q,qz/ak;q)_{\infty}}
{(qk/z,qa/z,kq,aq,q/a,q/k,qz/a,qz/k,q;q)_{\infty}}
\end{multline*}
where the last equality follows from \eqref{baileyeq1}, with $b=a$, $c=k$, $d=a/z$, $e=k/z$ and $ak/z$ instead of $a$.
Some further easy manipulations gives the final result.
\end{proof}

Remark: The proof that the sum of Lambert series above combine to give the stated
infinite product was first given by Andrews, Lewis and Liu in \cite{ALL01} (using a different labeling for the parameters) in a different context, so they did not have our reciprocity result for the basic hypergeometric series $f(a,k,z,q)$.

Note that substituting the expression for $f(a,k,z,q)$ from
\eqref{fakzqdef} into \eqref{wpeq7} leads to the identity
{\allowdisplaybreaks
\begin{multline}\label{wpeq7a}
\sum_{n=1}^{\infty} \frac{(q\sqrt{k},-q\sqrt{k},k,z,\frac{k}{a}
;q)_{n}(q;q)_{n-1}} {\left(\sqrt{k},-\sqrt{k},kq, q a,\frac{q
k}{z},q;q\right)_{n}}
\left( \frac{q a}{ z }\right )^{n} \\
- \sum_{n=1}^{\infty}
\frac{\left(\frac{q}{\sqrt{k}},\frac{-q}{\sqrt{k}},\frac{1}{k},\frac{1}{z},\frac{a}{k}
;q\right)_{n}(q;q)_{n-1}}
{\left(\frac{1}{\sqrt{k}},\frac{-1}{\sqrt{k}}, \frac{q }{k},\frac{q
z}{k},\frac{q}{a},q;q\right)_{n}}\left( \frac{q z}{ a }\right )^{n}
\\=
\frac{(a-k)\left(1-\frac{1}{z}\right)\left(1-\frac{ak}{z}\right)}
{(1-a)(1-k)\left(1-\frac{a}{z}\right)\left(1-\frac{k}{z}\right)}
+\frac{z}{k}\frac{\left(z,\frac{q}{z},\frac{k}{a},\frac{qa}{k},\frac{ak}{z},\frac{qz}{ak},q,q;q\right)_{\infty}}
{\left(\frac{z}{k},\frac{qk}{z},\frac{z}{a},\frac{qa}{z},a,\frac{q}{a},k,\frac{q}{k};q\right)_{\infty}},
\end{multline}
}an identity which does not appear to follow directly from Bailey's
formula at \eqref{baileyeq1}. We are now ready to prove Theorem
\ref{t1}.

\begin{proof}[Proof of Theorem \ref{t1}]
In the identity at \eqref{wpeq2}, note that the right side equals $f(a,k,z,q)$. Now replace $a$ with $1/a$, $k$ with $1/k$, $z$ with $1/z$, and subtract the resulting identity from the original identity. The left side of the resulting identity is the left side of \eqref{wpeq8}, while the right side is $f(a,k,z,q)-f(1/a,1/k,1/z,q)$, which by   \eqref{wpeq7} is the right side of \eqref{wpeq8}.
\end{proof}

Any WP-Bailey that is inserted into \eqref{wpeq8} will lead to a
summation formula for basic hypergeometric series. We give two
example as  illustrations.

\begin{corollary}
{\allowdisplaybreaks
\begin{multline}\label{wpeq9}
\sum_{n=1}^{\infty} \frac{(q\sqrt{k},-q\sqrt{k},z,k, \frac{k \rho_1}{a},\frac{k\rho_2}{a}, \frac{aq}{\rho_1\rho_2} ;q)_{n}(q;q)_{n-1}}
{\left(\sqrt{k},-\sqrt{k}, q k,\frac{q k}{z},\frac{aq}{\rho_1},\frac{aq}{\rho_2},\frac{k \rho_1 \rho_2}{a},q;q\right)_{n}}\left( \frac{q a}{ z }\right )^{n} \\
-
\sum_{n=1}^{\infty} \frac{\left(\frac{q}{\sqrt{k}},\frac{-q}{\sqrt{k}},\frac{1}{z},
\frac{1}{k},\frac{a\rho_1}{k},\frac{a\rho_2}{k},\frac{q}{a\rho_1\rho_2};q\right)_{n}(q;q)_{n-1}}
{\left(\frac{1}{\sqrt{k}},\frac{-1}{\sqrt{k}}, \frac{q }{k},\frac{q z}{k}
,\frac{q}{a\rho_1},\frac{q}{a\rho_2},\frac{a\rho_1\rho_2}{k},q;q\right)_{n}}\left( \frac{q z}{ a }\right )^{n}\\
-
 \sum_{n=1}^{\infty}\frac{(q\sqrt{a}, -q\sqrt{a},a,\rho_1,\rho_2,\frac{a^2q}{k\rho_1\rho_2},z;q)_{n}(q;q)_{n-1}}
 {\left(\sqrt{a},-\sqrt{a},\frac{aq}{\rho_1},\frac{aq}{\rho_2},\frac{k\rho_1\rho_2}{a},q a ,\frac{q
a}{z},q;q\right)_n}\left (\frac{q k}{z}\right)^{n}\\
+
\sum_{n=1}^{\infty}\frac{\left(\frac{q}{\sqrt{a}}, \frac{-q}{\sqrt{a}}, \frac{1}{a}, \rho_1,\rho_2,\frac{kq}{a^2\rho_1\rho_2},\frac{1}{z};q\right)_{n}(q;q)_{n-1}}
{\left(\frac{1}{\sqrt{a}},\frac{-1}{\sqrt{a}},\frac{q}{a\rho_1},\frac{q}{a\rho_2},\frac{a\rho_1\rho_2}{k},\frac{q}{a} ,\frac{q
z}{a},q;q\right)_n}\left (\frac{q z}{k}\right)^{n}  \\=
\frac{(a-k)\left(1-\frac{1}{z}\right)\left(1-\frac{ak}{z}\right)}
{(1-a)(1-k)\left(1-\frac{a}{z}\right)\left(1-\frac{k}{z}\right)}
+\frac{z}{k}\frac{\left(z,\frac{q}{z},\frac{k}{a},\frac{qa}{k},\frac{ak}{z},\frac{qz}{ak},q,q;q\right)_{\infty}}
{\left(\frac{z}{k},\frac{qk}{z},\frac{z}{a},\frac{qa}{z},a,\frac{q}{a},k,\frac{q}{k};q\right)_{\infty}}.
\end{multline}
}
\end{corollary}

\begin{proof}
Insert Singh's WP-Bailey pair \cite{S94},
%{\allowdisplaybreaks
\begin{align*}
%\label{singhpr}
\alpha_{n}(a,k)&=\frac{(q \sqrt{a}, -q
\sqrt{a},a,\rho_1,\rho_2,a^2q/k\rho_1\rho_2;q)_n}
{(\sqrt{a},-\sqrt{a},q,a q/\rho_1,a q/\rho_2,k\rho_1\rho_2/a;q)_n}\left(\frac{k}{a}\right)^n,\\
\beta_n(a,k)&=\frac{(k \rho_1/a, k\rho_2/a, k,
aq/\rho_1\rho_2;q)_n}{(a q/\rho_1, a q/\rho_2, k \rho_1
\rho_2/a,q;q)_n}, \notag
\end{align*}
%}
into \eqref{wpeq8}.
\end{proof}

\begin{corollary}
{\allowdisplaybreaks
\begin{multline}\label{wpeq90}
\sum_{n=1}^{\infty}
\frac{\left(q\sqrt{k},-q\sqrt{k},z,\frac{k^2}{qa^2};q\right)_{n}(q;q)_{n-1}}
{\left(\sqrt{k},-\sqrt{k}, q k,\frac{q k}{z},q;q\right)_{n}}
\left( \frac{q a}{ z }\right )^{n} \\
- \sum_{n=1}^{\infty}
\frac{\left(\frac{q}{\sqrt{k}},-\frac{q}{\sqrt{k}},\frac{1}{z},\frac{a^2}{qk^2};q\right)_{n}(q;q)_{n-1}}
{\left(\frac{1}{\sqrt{k}},-\frac{1}{\sqrt{k}}, \frac{q }{k},\frac{q
z}{k},q;q\right)_{n}}\left( \frac{q z}{ a }\right )^{n}\\
-\sum_{n=1}^{\infty}\frac{(q\sqrt{a}, -q\sqrt{a},a,z,\frac{k}{a
q};q)_{n}\left(\frac{q a^2}{k},q\right)_{2n}  (q;q)_{n-1}}
 {\left(\sqrt{a},-\sqrt{a},qa,\frac{q
a}{z},\frac{a^2q^2}{k},q;q\right)_n(k;q)_{2n}}\left (\frac{q k}{z}\right)^{n}\\
+\sum_{n=1}^{\infty}\frac{\left(\frac{q}{\sqrt{a}},
\frac{-q}{\sqrt{a}}, \frac{1}{a},
\frac{1}{z},\frac{a}{kq};q\right)_{n}\left(\frac{qk}{a^2},q\right)_{2n}
(q;q)_{n-1}}
{\left(\frac{1}{\sqrt{a}},\frac{-1}{\sqrt{a}},\frac{q}{a} ,\frac{q
z}{a},\frac{k q^2}{a^2},q;q\right)_n
\left(\frac{1}{k},q\right)_{2n}} \left (\frac{q z}{k}\right)^{n} \\=
\frac{(a-k)\left(1-\frac{1}{z}\right)\left(1-\frac{ak}{z}\right)}
{(1-a)(1-k)\left(1-\frac{a}{z}\right)\left(1-\frac{k}{z}\right)}
+\frac{z}{k}\frac{\left(z,\frac{q}{z},\frac{k}{a},
\frac{qa}{k},\frac{ak}{z},\frac{qz}{ak},q,q;q\right)_{\infty}}
{\left(\frac{z}{k},\frac{qk}{z},\frac{z}{a},\frac{qa}{z},a,\frac{q}{a},k,\frac{q}{k};q\right)_{\infty}}.
\end{multline}
}
\end{corollary}

\begin{proof}
Insert the WP-Bailey pair (see \cite[(3.3) - (3.4)]{AB02}
%{\allowdisplaybreaks
\begin{align*}
%\label{pr2}
\alpha_{n}(a,k)&=\frac{(q \sqrt{a}, -q \sqrt{a},a,k/aq;q)_n}
{(\sqrt{a},-\sqrt{a},q,a^2q^2/k;q)_n}\frac{(q a^2/k;q)_{2n}}{(k;q)_{2n}}\left(\frac{k}{a}\right)^n,\\
\beta_n(a,k)&=\frac{(k^2/qa^2;q)_n}{(q;q)_n}, \notag
\end{align*}
%}
into \eqref{wpeq8}.
\end{proof}

\section{Applications to Bailey Pairs}

If we let $k\to0$ in Lemmas \ref{p1}, \ref{l1} and
\ref{l2}, we get the following result.

\begin{theorem}\label{t2}
If $(\alpha_n, \beta_n)$ is a Bailey pair with respect to $a$, then subject to suitable convergence conditions,
{\allowdisplaybreaks
\begin{equation}\label{wpeq10}
\sum_{n=1}^{\infty} (z;q)_{n}(q;q)_{n-1}\left( \frac{q a}{ z }\right )^{n} \beta_n - \sum_{n=1}^{\infty}\frac{(z;q)_{n}(q;q)_{n-1}}{(q a ,q
a/z;q)_n}\left (\frac{q a}{z}\right)^{n}\alpha_n=f_1(a,z,q)
%\\
,
\end{equation}
}where
\begin{align*}
f_1(a,z,q)&=-\sum_{n=1}^{\infty} \frac{(q\sqrt{a},-q\sqrt{a},a,z;q)_{n}q^{n(n+1)/2}}
{(\sqrt{a},-\sqrt{a},q a,qa/z;q)_{n}(1-q^n)}\left( \frac{- a}{ z }\right )^{n}\\
&=\sum_{n=1}^{\infty} \frac{(z;q)_{n}}
{(q a;q)_{n}(1-q^n)}\left( \frac{q a}{ z }\right )^{n}\\
&=\sum_{n=1}^{\infty}\frac{a q^n/z}{1-a q^n/z}-\sum_{n=1}^{\infty}\frac{a q^n}{1-aq^n}.
\end{align*}
\end{theorem}
Note that the first two representations for $f_1(a,z,q)$ follow from \eqref{wpeq10}, upon inserting, respectively, the ``unit" Bailey pair
{\allowdisplaybreaks
\begin{align*}
%\label{up}
\alpha_{n}(a,q)&=\frac{(q \sqrt{a}, -q
\sqrt{a},a;q)_n}{(\sqrt{a},-\sqrt{a},q;q)_n}\left(-1\right)^n q^{n(n-1)/2},\\
\beta_n(a,q)&=\begin{cases} 1&n=0, \notag\\
0, &n>1,
\end{cases}
\end{align*}
}and the ``trivial" Bailey pair
\begin{align*}
%\label{tp}
\alpha_{n}(a,q)&=\begin{cases} 1&n=0, \\
0, &n>1,
\end{cases} \\
\beta_n(a,q)&=\frac{1}{(aq,q;q)_n}.\notag
\end{align*}
However, here and subsequently, we prefer to write these representations explicitly.
Upon letting $z \to \infty$ the following identity results.

\begin{corollary}\label{c4}
If $(\alpha_n, \beta_n)$ is a Bailey pair with respect to $a$, then
subject to suitable convergence conditions, {\allowdisplaybreaks
\begin{equation}\label{wpeq11}
\sum_{n=1}^{\infty} (q;q)_{n-1}\left( - a\right )^{n}q^{n(n+1)/2}
\beta_n - \sum_{n=1}^{\infty}\frac{(q;q)_{n-1}\left (-
a\right)^{n}q^{n(n+1)/2}}{(q a ;q)_n}\alpha_n=f_2(a,q)
%\\
,
\end{equation}
}where
\begin{align}\label{f2eq}
f_2(a,q)&=-\sum_{n=1}^{\infty} \frac{(1-a
q^{2n})q^{n^2} a^{n}}
{(1-a q^{n})(1-q^n)}\\
&=\sum_{n=1}^{\infty} \frac{q^{n(n+1)/2}\left(-a\right)^{n}}
{(q a;q)_{n}(1-q^n)}\notag\\
&=-\sum_{n=1}^{\infty}\frac{a q^n}{1-aq^n}.\notag
\end{align}
\end{corollary}

As is well known, many theta products/series can be represented as
sums of Lambert series of the type immediately above. The other
representations of $f_2(a,q)$ now let these theta functions be
represented in two different ways as basic hypergeometric series.

Let \[ a(q):=\sum_{m,n=-\infty}^{\infty}q^{m^2+mn+n^2}.
\]Here we are using the notation  for this series employed in \cite{BB91}.
\begin{corollary}\label{corlam1}
{\allowdisplaybreaks
\begin{align}\label{corlameq1}
a(q)&=1-6\sum_{n=1}^{\infty}\frac{(-1)^n
q^{(3n^2-n)/2}}{(q;q^3)_n(1-q^{3n})}+6\sum_{n=1}^{\infty}\frac{(-1)^n
q^{(3n^2+n)/2}}{(q^2;q^3)_n(1-q^{3n})},\\
&=1+6\sum_{n=1}^{\infty}\frac{
q^{3n^2-2n}(1-q^{6n-2})}{(1-q^{3n-2})(1-q^{3n})}
-6\sum_{n=1}^{\infty}\frac{
q^{3n^2-n}(1-q^{6n-1})}{(1-q^{3n-1})(1-q^{3n})}.\notag
\end{align}
}
\end{corollary}
\begin{proof}
The following result is
\textbf{Entry 18.2.8} of Ramanujan's Lost Notebook (see
\cite[page 402]{AB05}):
\begin{align}\label{corlameq2}
a(q)&=1+6
\sum_{n=1}^{\infty}\left(\frac{n}{3}\right)\frac{q^n}{1-q^n}\\
&=1+6\sum_{n=1}^{\infty}\frac{q^{-2}q^{3n}}{1-q^{-2}q^{3n}}
-6\sum_{n=1}^{\infty}\frac{q^{-1}q^{3n}}{1-q^{-1}q^{3n}}.\notag
\end{align}
Use \eqref{f2eq} (with $q$ replaced with $q^3$ and $a=q^{-1}$ and
$a=q^{-2}$, respectively) to replace each of the Lambert series
with, in turn, each of the other two representations of $f_2(a,q^3)$,
and the result follows.
\end{proof}

Remark: It is clear that a quite general statement concerning
$a(q)$ may be deduced from \eqref{corlameq1} by a similar argument.
Indeed, if $(\alpha_n(a,q)$,  $\beta_n(a,q))$ is \emph{any} Bailey
pair in which $a$ is a free parameter, then
\begin{multline}\label{corlameq3}
a(q)=1+6\sum_{n=1}^{\infty} (q^3;q^3)_{n-1}\left( - 1\right
)^{n}q^{(3n^2+n)/2} \beta_n(q^{-1},q^3)\\
-6\sum_{n=1}^{\infty}
(q^3;q^3)_{n-1}\left( - 1\right )^{n}q^{(3n^2-n)/2}
\beta_n(q^{-2},q^3)\\
-6\sum_{n=1}^{\infty}\frac{(q^3;q^3)_{n-1}\left (-
1\right)^{n}q^{(3n^2+n)/2}\alpha_n(q^{-1},q^3)}{(q^2
;q^3)_n}\\
+6\sum_{n=1}^{\infty}\frac{(q^3;q^3)_{n-1}\left (-
1\right)^{n}q^{(3n^2-n)/2}\alpha_n(q^{-2},q^3)}{(q ;q^3)_n}.
\end{multline}
As an example, if we insert the Bailey pair of Slater \cite[Equation
(4.1), page 469]{S51},
\begin{align*}
%\label{wpS}
\alpha_n(a,q)&=\frac{(1-a q^{2n})(a,c,d;q)_n}{(1-a)(aq/c,aq/d,
q;q)_n}\left( \frac{-a}{cd}\right)^n q^{(n^2+n)/2},\\
\beta_n(a,q)&=\frac{(aq/cd;q)_n}{(aq/c,aq/d,q;q)_n}, \notag
\end{align*}
in \eqref{corlameq3}, we get, for any values for $c$ and $d$ that do
not make any denominator vanish, that {\allowdisplaybreaks
\begin{multline}\label{aqcdeq}
a(q)=1+6\sum_{n=1}^{\infty}
\frac{(q^3;q^3)_{n-1}(q^2/cd;q^3)_n\left( - 1\right
)^{n}q^{(3n^2+n)/2}}{(q^2/c,q^2/d,q^3;q^3)_n} \\
-6\sum_{n=1}^{\infty} \frac{(q^3;q^3)_{n-1}(q/cd;q^3)_n\left( -
1\right
)^{n}q^{(3n^2-n)/2}}{(q/c,q/d,q^3;q^3)_n}\\
-6\sum_{n=1}^{\infty}\frac{(1-q^{6n-1})(q^3;q^3)_{n-1}(1/q,c,d;q^3)_n
q^{3n^2+n}}{(1-1/q)(q^2/c,q^2/d,q^2,q^3
;q^3)_n c^n d^n}\\
+6\sum_{n=1}^{\infty}\frac{(1-q^{6n-2})(q^3;q^3)_{n-1}(1/q^2,c,d;q^3)_n
q^{3n^2-n}}{(1-1/q^2)(q/c,q/d,q,q^3
;q^3)_n c^n d^n}.
\end{multline}
}

If we let $c, d\to \infty$ in this identity we get that
\begin{multline}
a(q)=1+6\sum_{n=1}^{\infty} \frac{\left( - 1\right
)^{n}q^{(3n^2+n)/2}}{1-q^{3n}}  -6\sum_{n=1}^{\infty} \frac{\left( -
1\right
)^{n}q^{(3n^2-n)/2}}{1-q^{3n}}\\
-6\sum_{n=1}^{\infty}\frac{(1-q^{6n-1})
q^{6n^2-2n}}{(1-q^{3n-1})(1-q^{3n})}
+6\sum_{n=1}^{\infty}\frac{(1-q^{6n-2})
q^{6n^2-4n}}{(1-q^{3n-2})(1-q^{3n})}.
\end{multline}

A similar situation will hold for some of the other identities given
below. Recall (see \cite[page 36]{B91})
\begin{equation}\label{thetaeq}
\phi(q):=\sum_{n=-\infty}^{\infty}q^{n^2}=(-q,-q,q^2;q^2)_{\infty}.
\end{equation}

\begin{corollary}\label{corlam2}
\begin{align}
\phi(q)^2&=1-4\sum_{n=1}^{\infty}\frac{(-1)^nq^{2n^2-n}}{(q;q^4)_n(1-q^{4n})}
+ 4\sum_{n=1}^{\infty}\frac{(-1)^nq^{2n^2+n}}{(q^3;q^4)_n(1-q^{4n})},\\
&=1+4\sum_{n=1}^{\infty}\frac{(1-q^{8n-3})q^{4n^2-3n}}{(1-q^{4n-3})(1-q^{4n})}
-4\sum_{n=1}^{\infty}\frac{(1-q^{8n-1})q^{4n^2-n}}{(1-q^{4n-1})(1-q^{4n})}.\notag
\end{align}
For any values for $c$ and $d$ that do not make any
denominator vanish,
\begin{multline}\label{corlam2eq2}
\phi(q)^2=1+4\sum_{n=1}^{\infty}
\frac{(q^4;q^4)_{n-1}(q^3/cd;q^4)_n\left( - 1\right
)^{n}q^{2n^2+n}}{(q^3/c,q^3/d,q^4;q^4)_n} \\
-4\sum_{n=1}^{\infty} \frac{(q^4;q^4)_{n-1}(q/cd;q^4)_n\left( -
1\right
)^{n}q^{2n^2-n}}{(q/c,q/d,q^4;q^4)_n}\\
-4\sum_{n=1}^{\infty}\frac{(1-q^{8n-1})(q^4;q^4)_{n-1}(1/q,c,d;q^4)_n
q^{4n^2+2n}}{(1-1/q)(q^3/c,q^3/d,q^3,q^4
;q^4)_n c^n d^n}\\
+4\sum_{n=1}^{\infty}\frac{(1-q^{8n-3})(q^4;q^4)_{n-1}(1/q^3,c,d;q^4)_n
q^{4n^2-2n}}{(1-1/q^3)(q/c,q/d,q,q^4
;q^4)_n c^n d^n}.
\end{multline}
\end{corollary}

\begin{proof}
By \textbf{Entry 8 (i)} in chapter 17 of \cite{B91},
\[
\phi(q)^2=1+4 \sum_{n=1}^{\infty}\frac{q^{4n-3}}{1-q^{4n-3}}-4 \sum_{n=1}^{\infty}\frac{q^{4n-1}}{1-q^{4n-1}}.
\]
We omit the remainder of the arguments, since they parallel those for the identities involving $a(q)$ above.
\end{proof}

If we let $c,d\to \infty$ in \eqref{corlam2eq2}, we get the identity
\begin{multline}\label{corlam2eq22}
\phi(q)^2=1+4\sum_{n=1}^{\infty} \frac{\left( - 1\right
)^{n}q^{2n^2+n}}{1-q^{4n}}  -4\sum_{n=1}^{\infty} \frac{\left( -
1\right
)^{n}q^{2n^2-n}}{1-q^{4n}}\\
-4\sum_{n=1}^{\infty}\frac{(1-q^{8n-1})
q^{8n^2-2n}}{(1-q^{4n-1})(1-q^{4n})}
+4\sum_{n=1}^{\infty}\frac{(1-q^{8n-3})
q^{8n^2-6n}}{(1-q^{4n-3})(1-q^{4n})}.
\end{multline}

\begin{corollary}\label{cz-1}
If $(\alpha_n, \beta_n)$ is a Bailey pair with respect to $a$, then subject to suitable convergence conditions,
{\allowdisplaybreaks
\begin{equation}\label{wpeq12}
\sum_{n=1}^{\infty} (q^2;q^2)_{n-1}\left( -q a\right )^{n} \beta_n - \sum_{n=1}^{\infty}\frac{(q^2;q^2)_{n-1}\left (-q a\right)^{n}}{(q^2 a^2 ;q^2)_n}\alpha_n=f_3(a,q)
%\\
,
\end{equation}
}where
\begin{align*}
f_3(a,q)&=-\sum_{n=1}^{\infty} \frac{(q\sqrt{a},-q\sqrt{a},a;q)_{n}(-q;q)_{n-1}q^{n(n+1)/2}a^{n}}
{(\sqrt{a},-\sqrt{a};q)_{n}(q^2 a^2;q^2)_{n}(1-q^n)} \\
&=\sum_{n=1}^{\infty} \frac{(-q;q)_{n-1}\left( -q a\right )^{n}}
{(q a;q)_{n}(1-q^n)}\\
&=-\sum_{n=1}^{\infty}\frac{a q^n}{1-a^2q^{2n}}.
\end{align*}
\end{corollary}
\begin{proof}
Let $z \to -1$ in \eqref{wpeq10} and simplify.
\end{proof}
One reason we single out this special case is that many theta products/series can also be
expressed in terms of Lambert series of the type just above. We consider one example.
Recall (see \cite[page 36]{B91}) that
\[
\psi (q):=\sum_{n=0}^{\infty}q^{n(n+1)/2}=\frac{(q^2;q^2)_{\infty}}{(q;q^2)_{\infty}.}
\]
By \textbf{Entry 34 (p.284)} in chapter 36 of Ramanujan's notebooks (see \cite[page 374]{B98}),
\begin{equation}\label{psieq1}
q\frac{\psi^3(q^3)}{\psi(q)}=\sum_{n=1}^{\infty}\frac{q^{3n-2}}{1-q^{6n-4}} -
\sum_{n=1}^{\infty}\frac{q^{3n-1}}{1-q^{6n-2}}.
\end{equation}
Upon replacing $q$ with $q^3$ and $a$ with $q^{-2}$ and then $a$ with $q^{-1}$ in Corollary \ref{cz-1} and combining the various serious appropriately, we get the following identities.
\begin{corollary}\label{cpsi}
\begin{align}\label{cpsieq1}
q\frac{\psi^3(q^3)}{\psi(q)}&=\sum_{n=1}^{\infty}
\frac{(1-q^{6n-2})(-q^3;q^3)_{n-1}q^{(3n^2-n)/2}}{(1-q^{3n-2})(1-q^{3n})(-q;q^3)_n}\\
&\phantom{asdasdasdasdas}
-\sum_{n=1}^{\infty}
\frac{(1-q^{6n-1})(-q^3;q^3)_{n-1} q^{(3n^2+n)/2}}{(1-q^{3n-1})(1-q^{3n})(-q^2;q^3)_n},\notag\\
&=\sum_{n=1}^{\infty}
\frac{(-q^3;q^3)_{n-1}(-1)^n q^{2 n}}{(1-q^{3n})(q^2;q^3)_n}
%\\
%&\phantom{asdasdasdasdas}
-\sum_{n=1}^{\infty}
\frac{(-q^3;q^3)_{n-1}(-1)^n q^{n}}{(1-q^{3n})(q;q^3)_n}.\notag
\end{align}
If $(\alpha_n(a,q), \beta_n(a,q))$ is a Bailey pair in which $a$ is a free parameter, then
\begin{multline}\label{psieq2}
q\frac{\psi^3(q^3)}{\psi(q)}\\=\sum_{n=1}^{\infty}(q^6;q^6)_{n-1}(-q^2)^n \beta_n(1/q,q^3)
-\sum_{n=1}^{\infty}(q^6;q^6)_{n-1}(-q)^n \beta_n(1/q^2,q^3)\\
+\sum_{n=1}^{\infty}\frac{(q^6;q^6)_{n-1}(-q)^n }{(q^2;q^6)_n}\alpha_n(1/q^2,q^3)
-\sum_{n=1}^{\infty}\frac{(q^6;q^6)_{n-1}(-q^2)^n }{(q^4;q^6)_n}\alpha_n(1/q,q^3).
\end{multline}
If $d \not=q^{3n\pm1}$, then
{\allowdisplaybreaks\begin{multline}\label{psieq3}
q\frac{\psi^3(q^3)}{\psi(q)}=
\sum_{n=1}^{\infty}\frac{(q^6;q^6)_{n-1}(q^2/d;q^6)_n(-q^2)^n }
{(q^2;q^6)_n(q^2/d,q^3;q^3)_n}\\
-\sum_{n=1}^{\infty}\frac{(q^6;q^6)_{n-1}(q/d;q^6)_n(-q)^n }
{(q;q^6)_n(q/d,q^3;q^3)_n}\\
+\sum_{n=1}^{\infty}\frac{(1-q^{12n-2})(q^6;q^6)_{2n-1}(1/q^2,d;q^6)_n q^{6n^2} }{(1-1/q^2)(q^2;q^6)_{2n}(q^4/d,q^6;q^6)_n d^n}\\
-\sum_{n=1}^{\infty}\frac{(1-q^{12n-1})(q^6;q^6)_{2n-1}(1/q,d;q^6)_n q^{6n^2+3n} }{(1-1/q)(q^4;q^6)_{2n}(q^5/d,q^6;q^6)_n d^n}.
\end{multline}}
\begin{multline}\label{psieq4}
q\frac{\psi^3(q^3)}{\psi(q)}=
\sum_{n=1}^{\infty}\frac{(q^6;q^6)_{n-1}(-q^2)^n }
{(q^2;q^6)_n(q^3;q^3)_n}
-\sum_{n=1}^{\infty}\frac{(q^6;q^6)_{n-1}(-q)^n }
{(q;q^6)_n(q^3;q^3)_n}\\
+\sum_{n=1}^{\infty}\frac{(1-q^{12n-2})(q^6;q^6)_{2n-1}(1/q^2;q^6)_n (-1)^n q^{9n^2-3n} }{(1-1/q^2)(q^2;q^6)_{2n}(q^6;q^6)_n}\\
-\sum_{n=1}^{\infty}\frac{(1-q^{12n-1})(q^6;q^6)_{2n-1}(1/q;q^6)_n (-1)^n q^{9n^2} }{(1-1/q)(q^4;q^6)_{2n}(q^6;q^6)_n}.
\end{multline}
\end{corollary}
\begin{proof}
Identities \eqref{cpsieq1} - \eqref{psieq2} follow directly from Corollary \ref{cpsi}.
The identity at \eqref{psieq3} follows upon inserting the Bailey pair
(see Corollary 2.13 in \cite{MSZ09})
\begin{align*}
%\label{sgen660}
\alpha_{2r}&=
\frac{\displaystyle{1-a q^{4r}}}{\displaystyle{1-a}}
\frac{\left(\displaystyle{a,d;q^2}\right)_r a^r\displaystyle{q^{2r^2}}}
   {\left(\displaystyle{a q^2/d,q^2;q^2}\right)_r\displaystyle{d^r}},\\
\alpha_{2r-1}&=0,\notag
\\
\beta_n&= \frac{(a q/d;q^2)_n}{(a q;q^2)_{n}(aq/d,q;q)_n},\,\, \text{ with respect to $a=a$}\notag
\end{align*}
into \eqref{psieq2}, and \eqref{psieq4} is a consequence of letting $d \to \infty$ in \eqref{psieq3}.
\end{proof}

\section{The Lambert series $\sum_{n=1}^{\infty} \frac{a q^n}{1-a q^n}$ again}
Define
\begin{equation*}
L_a(q):=\sum_{n=1}^{\infty} \frac{a q^n}{1-a q^n}.
\end{equation*}
From Corollary \ref{c4} it can be seen that $L_a(q)$ can be
variously represented as {\allowdisplaybreaks\begin{align}
\label{laeq} &L_a(q)=\sum_{n=1}^{\infty} \frac{(1-a q^{2n})q^{n^2}
a^{n}}
{(1-a q^{n})(1-q^n)}\\
&=-\sum_{n=1}^{\infty} \frac{q^{n(n+1)/2}\left(-a\right)^{n}}
{(q a;q)_{n}(1-q^n)}\notag\\
&=\sum_{n=1}^{\infty}\frac{(q;q)_{n-1}\left (-
a\right)^{n}q^{n(n+1)/2}}{(q a ;q)_n}\alpha_n-\sum_{n=1}^{\infty} (q;q)_{n-1}\left( - a\right )^{n}q^{n(n+1)/2}
\beta_n, \notag
\end{align}}
where $(\alpha_n, \beta_n)$ is a Bailey pair with respect to $a$.
We subsequently noticed that it was possible to give two additional
representations of $L_a(q)$.
\begin{corollary}\label{claq2}
\begin{align}\label{laeq2}
L_a(q)&=\frac{-1}{(aq;q)_{\infty}}\sum_{n=1}^{\infty}\frac{n (-a)^n q^{n(n+1)/2}}{(q;q)_n},\\
&=(aq;q)_{\infty}\sum_{n=1}^{\infty}\frac{n a^n q^n}{(q;q)_n}. \notag
\end{align}
\end{corollary}
\begin{proof}
Let $k \to 0$ and $z \to \infty$ in \eqref{wpeq}, and rearrange to get
\begin{multline}\label{peq}
\sum_{n=1}^{\infty} ( y q;q)_{n-1}
\left( \frac{-a}{y}\right )^{n}q^{n(n+1)/2} \beta_n -\\
\frac{(q a/y;q)_{\infty}} {(q a;q)_{\infty}} \sum_{n=1}^{\infty}\frac{(yq;q)_{n-1}}{(q a/y ;q)_n}\left (\frac{-a}{y }\right)^{n}q^{n(n+1)/2}\alpha_n\\=
\frac{1}{1-y}\left( \frac{(qa/y;q)_{\infty}} {(q a;q)_{\infty}} -1\right),
\end{multline}
where $(\alpha_n,\beta_n)$ is a Bailey pair with respect to $a$. The result of letting $y\to 1$ on the left side of \eqref{peq} is the left side of \eqref{wpeq11}, and hence equals
$-L_a(q)$, from the final representation of $f_2(a,q)$ in Corollary \ref{c4}.

Thus \eqref{peq} now gives that
\begin{align*}
L_a(q)&=-\lim_{y \to 1}\frac{1}{1-y}\frac{\left( (q a/y;q)_{\infty}-(q a;q)_{\infty}\right)}{(q a;q)_{\infty}} \\
&=- \lim_{y \to 1}\frac{(q a/y;q)_{\infty}}{1-y}\left( \frac{1}{(q a;q)_{\infty}} -\frac{1}{(q a/y;q)_{\infty}}\right).
\end{align*}

In the first case we use the second identity at \eqref{bineqab} to get that
\begin{align*}
L_a(q)&=-\lim_{y \to 1}\frac{1}{1-y}\frac{\left( (q a/y;q)_{\infty}-(q a;q)_{\infty}\right)}{(q a;q)_{\infty}} \\
&=- \frac{1}{(q a;q)_{\infty}}\sum_{n=1}^{\infty}\frac{q^{n(n-1)/2}}{(q;q)_n} \lim_{y \to 1}\frac{(-qa/y)^n-(-qa)^n}{1-y},\\
&=\frac{-1}{(aq;q)_{\infty}}\sum_{n=1}^{\infty}\frac{n (-a)^n q^{n(n+1)/2}}{(q;q)_n},
\end{align*}
by L'Hospital's rule. The proof for the other representation follows similarly, using the first identity at \eqref{bineqab}.
\end{proof}

These expressions for $L_a(q)$ may also be used to write any theta product that is expressible in terms of such Lambert series in terms of $q$-series similar to those in Corollary \ref{claq2}. Recall that $\phi(q)$ is defined at \eqref{thetaeq}.

\begin{corollary}
\begin{align}
\frac{\theta^3(-q)}{\theta(-q)}&=1-\frac{6}{(-q;q^3)_{\infty}}\sum_{n=1}^{\infty} \frac{n q^{(3n^2+n)/2}}{(q^3;q^3)_n}+\frac{6}{(-q^2;q^3)_{\infty}}\sum_{n=1}^{\infty} \frac{n q^{(3n^2-n)/2}}{(q^3;q^3)_n},\\
&=1+6(-q;q^3)_{\infty}\sum_{n=1}^{\infty} \frac{n (-1)^n q^{n}}{(q^3;q^3)_n}-6(-q^2;q^3)_{\infty}\sum_{n=1}^{\infty} \frac{n (-1)^n q^{2n}}{(q^3;q^3)_n}. \notag
\end{align}
\end{corollary}
\begin{proof}
By \textbf{Entry 18.2.16 (formula (1.21), p.353; formula (3.51),}\\
\textbf{p.357)} in Ramanujan's Lost Notebook (see \cite[page 405]{AB05}),
\begin{equation}\label{thetaeq2}
\frac{\phi^3(-q)}{\phi(-q^3)}=1-6\sum_{n=1}^{\infty}\frac{q^{3n-2}}{1+q^{3n-2}}
+6\sum_{n=1}^{\infty}\frac{q^{3n-1}}{1+q^{3n-1}}.
\end{equation}
The proofs now follow as a consequence of Corollary \ref{claq2}, with $q$ replaced with $q^3$, and $a$ taking, in turn, the values $-1/q^2$ and $-1/q$.
\end{proof}

Remark: Ramanujan gives a number of other examples of theta products
expressible as sums of Lambert series of the types considered in the
present paper. The methods of the present paper could also be
applied to those theta products, but we refrain from further
examples, leaving these for the reader's own entertainment.

 \allowdisplaybreaks{

}
\end{document}